\documentclass[11pt,reqno]{amsart}

\usepackage[T1]{fontenc}
\usepackage{lmodern}
\usepackage{microtype}
\usepackage{mathtools,amssymb,aliascnt}
\usepackage[margin=1.12in]{geometry}
\usepackage[colorlinks=true,linkcolor=blue,citecolor=blue,urlcolor=blue]{hyperref}
\usepackage[capitalise,noabbrev]{cleveref}

\newtheorem{theorem}{Theorem}[section]

\newaliascnt{lemma}{theorem}
\newtheorem{lemma}[lemma]{Lemma}
\aliascntresetthe{lemma}
\crefname{lemma}{lemma}{lemmas}
\Crefname{lemma}{Lemma}{Lemmas}
\theoremstyle{definition}
\newtheorem*{definition}{Definition}
\newtheorem{problem}{Problem}
\crefname{problem}{problem}{problems}
\Crefname{problem}{Problem}{Problems}
\theoremstyle{plain}

\newcommand{\R}{\mathbb R}

\newcommand{\E}{\mathbb E}
\newcommand{\PP}{\mathbb P}
\DeclareMathOperator{\Var}{Var}
\DeclareMathOperator{\Lip}{Lip}
\DeclareMathOperator{\arsinh}{arsinh}
\numberwithin{equation}{section}
\title[Property $(H)$ of $c_0$ and Lipschitz problems]
{On Property (H) of $c_0$ and Lipschitz Problems of Gromov and Johnson}
\author{Lixin Cheng}
\thanks{The first author was Supported by the National Natural Science Foundation of China
(NSFC), Grant No.~12271453.}
\address{School of Mathematical Sciences, Xiamen University,
Xiamen 361005, China}
\email{lxcheng@xmu.edu.cn (L. Cheng)}

\author{Qingjin Cheng}
\thanks{The second author was Supported by the National Natural Science Foundation of China
(NSFC), Grant No.~12071389.}
\address{School of Mathematical Sciences, Xiamen University,
Xiamen 361005, China}
\email{qjcheng@xmu.edu.cn (Q. Cheng)}

\author{Yue Wang}
\address{School of Mathematical Sciences, Xiamen University,
Xiamen 361005, China}
\email{19020260158074@stu.xmu.edu.cn (Y. Wang)}

\date{}
\hypersetup{
 pdftitle={On Property (H) of c0 and Lipschitz Problems of Gromov and Johnson},
 pdfauthor={L. Cheng, Q. Cheng and Y. Wang},
 pdfsubject={Nonzero degree, common moduli, and sharp Lipschitz bounds}
}

\begin{document}
\begin{abstract}
Let $X$ be a normed space and  $S_X$ be the unit sphere of $X$. We say that $X$ has Property (H) if there exist two increasing sequences $V_j\subset X$,\;$H_j\subset \ell_2, j=1,2,\cdots$  of finite dimensional subspaces  satisfying the union $V\equiv\bigcup_j V_j$ is dense in $X$ and there is a uniform continuous mapping $\psi: S_V\rightarrow S_{\ell_2}$ so that the restrictions $\psi|_{S_{V_j}}, j=1,2,\cdots$ are isomorphisms from $S_{V_j}$ to $S_{H_j}$. The Kasparov-Yu problem asks whether the Banach space $c_0$ of all null sequences admits Property (H), or rational Property (H). In this paper, we prove that $c_0$ admits neither Property (H) nor rational Property (H). We also obtain the following results.

\begin{enumerate}
\item[(1)]

The infimum $\mu_n$ of the Lipschitz constants of all
nonzero-degree maps from $S_{\ell_\infty^n}$ to $S_{\ell_2^n}$
satisfies
\[
\lim\limits_{n\to\infty}\frac{\mu_n}{\log n}=\frac12.
\]

\item[(2)]

Let $\lambda_n$ and $C_n$ denote the infima of the Lipschitz
constants of homeomorphisms from $S_{\ell_\infty^n}$ onto
$S_{\ell_2^n}$ and from $B_{\ell_\infty^n}$ onto $B_{\ell_2^n}$,
respectively. Then
\[
 \lim_{n\to\infty}\frac{\lambda_n}{\log n}
 =\lim_{n\to\infty}\frac{C_n}{\log n}=\frac12.
\]
\end{enumerate}

The first result solves a problem of Gromov on nonzero-degree maps between
finite-dimensional spheres, while the second settles Johnson's problem on
the Lipschitz distortion of finite-dimensional unit balls.
\end{abstract}

\maketitle

\section{Introduction and main results}

This paper answers three problems in the quantitative geometry of
finite-dimensional unit spheres and balls. Although they arise in
different forms, all three ask
how topological information on finite-dimensional spheres interacts
with quantitative metric control. We first give a negative answer to
the Kasparov--Yu problem for $c_0$. We then turn to the distinct but
closely related Lipschitz problems posed by Gromov and Johnson and
determine the sharp logarithmic asymptotics of the relevant Lipschitz
constants, including Gromov's stabilized formulation. A common
degree--Poincar\'e method underlies these results.

\subsection{The Kasparov--Yu problem}
Kasparov and Yu introduced Property~$(H)$ in their work on the Novikov
conjecture \cite{KY}. They proved that a countable discrete group
which coarsely embeds into a Banach space with this property satisfies
the strong Novikov conjecture. We refer also to \cite{CWY} for the
corresponding result on the coarse Novikov conjecture.

All Banach spaces in this paper are real. We write $B_X$ and $S_X$ for
the closed unit ball and unit sphere of a normed space $X$, endowed with
the norm distances inherited from $X$.  A \emph{paving} of $X$ is an increasing sequence
$(V_j)$ of finite-dimensional subspaces whose union is dense in $X$.

\begin{definition}[\cite{KY}]
An infinite-dimensional real Banach space $X$ has \emph{rational Property~$(H)$} if there exist
a paving $(V_j)$ of $X$ and a paving $(H_j)$ of $\ell_2$ with
$\dim V_j=\dim H_j\ge2$, and a uniformly continuous map $\psi:S_V\longrightarrow S_{\ell_2}$, where
$V=\bigcup_j V_j$, such that
$\psi|_{S_{V_j}}:S_{V_j}\longrightarrow S_{H_j}$ has nonzero degree
for every $j$.
\end{definition}

In the terminology of Kasparov and Yu, Property~$(H)$ requires each restriction to be a
homeomorphism onto $S_{H_j}$, or more generally, a map of degree one.

The classical Mazur maps give Property~$(H)$ for $\ell_p$, $1\le p<\infty$,
and their noncommutative counterparts give it for Schatten classes
\cite[pp.~1859--1860]{KY}. Further examples were obtained using sphere
geometry and Banach lattice methods; see \cite{ChengH,CW}.
Every countable discrete group, equipped with a proper left-invariant
metric, coarsely embeds into $c_0$ by Aharoni's theorem \cite{Aharoni}. Kasparov and Yu asked the following questions; see also
\cite[Problem~1.1]{BGLMPS} for the Property~$(H)$ question.

\begin{problem}[Kasparov and Yu {\cite[p.~1861]{KY}}]\label{prob:KY}
Does $c_0$ have Property~$(H)$? Does it have rational Property~$(H)$?
\end{problem}

Our first result answers both questions negatively.

\begin{theorem}\label{thm:main}
The space $c_0$ has neither rational Property~$(H)$ nor Property~$(H)$.
\end{theorem}

\subsection{The Lipschitz problems of Gromov and Johnson}

We next turn to two finite-dimensional Lipschitz problems posed
independently by Gromov and Johnson. Gromov considered maps of
nonzero degree between the unit spheres of $\ell_\infty^n$ and
$\ell_2^n$, whereas Johnson asked about homeomorphisms between the
corresponding closed unit balls. Although the two problems have
different origins, they are linked by their sphere formulations and,
as we shall see, by the same sharp logarithmic asymptotic.

In his discussion of positive scalar curvature and the Novikov
conjecture, Gromov posed the following question in
\emph{Spaces and questions} \cite[Section~II(b)]{Gromov}.
We reproduce it with the notation and sphere terminology
adapted to our conventions.\footnote{Our sphere indices refer
to the dimensions of the underlying vector spaces. In the
stabilized version, we interpret the map as having nonzero
degree between manifolds of the same dimension: the second
sphere factor has dimension $M$, and both the domain and
the range have dimension $N+M-1$.}
\begin{quote}
It appears that an essential part of the difficulty in understanding
$\mathbf{S}>0$ (and the Novikov conjecture) is linked to the following
simple minded question:
\emph{what is the minimal $\lambda>0$, such that the unit sphere
$S_{\ell_\infty^N}$ in the Banach space
$\ell_\infty^N=(\R^N,\|x\|_\infty=\sup_i|x_i|)$ admits a
$\lambda$-Lipschitz map into the Euclidean unit sphere $S_{\ell_2^N}$ in
$\R^N=\ell_2^N$ with non-zero degree}?
Probably, $\lambda\to\infty$ for $N\to\infty$ (even if we stabilize
to maps $S_{\ell_\infty^N}\times S_{\ell_2^{M+1}}(r)\to S_{\ell_2^{N+M}}$ with
arbitrarily large $M$ and $r$) and this might indicate new ways of
measuring ``size of $V$'' in the context of $\mathbf{S}>0$ and the
Novikov conjecture.
\end{quote}

For a normed space $E$ and $r>0$, write $B_E(r)=\{x\in E:\|x\|_E\le r\},  S_E(r)=\{x\in E:\|x\|_E=r\}$.
We abbreviate $B_E(1)$ and $S_E(1)$ to $B_E$ and $S_E$.
Thus $S_{\ell_2^k}(r)$ is the Euclidean sphere in $\R^k$
of radius $r$ centered at the origin; its dimension is $k-1$.
All spheres carry the distances induced by their respective norms.

For a map $f:\mathcal{M}\to\mathcal{N}$ between metric spaces
$(\mathcal{M},d_{\mathcal{M}})$ and
$(\mathcal{N},d_{\mathcal{N}})$, define
\[
 \Lip(f)=\sup_{x\ne y}
 \frac{d_{\mathcal{N}}(f(x),f(y))}
      {d_{\mathcal{M}}(x,y)}.
\]
On $S_{\ell_\infty^n}\times S_{\ell_2^{m+1}}(r)$,
we use the maximum product metric
\[
 d_\infty((x,u),(x',u'))=
 \max\{\|x-x'\|_\infty,\|u-u'\|_2\}.
\]
With these conventions, we state the original and stabilized
versions of Gromov's sphere problem as follows.

\begin{problem}[Gromov {\cite[Section~II(b)]{Gromov}}]\label{prob:gromov}
Do the least Lipschitz constants of nonzero-degree maps
$S_{\ell_\infty^n}\to S_{\ell_2^n}$ tend to infinity as
$n\to\infty$? Does the same conclusion remain valid after
stabilization, when the dimension and radius of the stabilizing
Euclidean sphere may vary with $n$?
\end{problem}

Johnson asked the corresponding boundedness question for
homeomorphisms between the closed unit balls.
\begin{problem}[Johnson {\cite{Johnson}}]\label{prob:johnson}
Is there a sequence of homeomorphisms
$g_n:B_{\ell_\infty^n}\to B_{\ell_2^n}$, $n\ge2$, whose
Lipschitz constants are uniformly bounded?
\end{problem}

The ball and sphere formulations are equivalent for boundedness
independent of dimension: a homeomorphism between closed unit balls
of finite-dimensional normed spaces maps their boundaries onto each
other, and a Lipschitz homeomorphism between the spheres in
the corresponding sphere formulation extends by positive homogeneity to a homeomorphism
of their balls, with the bound in \Cref{lem:radial}.
No control of the inverse is required. The sphere formulation is also
used in \cite[Problem~1.2]{BGLMPS}. For the sharp asymptotic in
Johnson's ball problem, however, we shall use a separate
interior construction.

Recently, Braga, Gartland, Lancien, Motakis, Perneck\'a and Schlumprecht
\cite{BGLMPS} studied common moduli of uniform continuity for
homeomorphisms from $S_{\ell_\infty^n}$ to $S_{\ell_1^n}$. They excluded
such families when the maps do not increase support sizes or are step
preserving. Here we rule out a common modulus for all nonzero-degree
maps into Euclidean spheres, without either assumption. This also
excludes a common modulus for homeomorphisms into $S_{\ell_1^n}$:
composition with the classical Mazur homeomorphism
$S_{\ell_1^n}\to S_{\ell_2^n}$, whose $1/2$-H\"older constant is
independent of $n$, would give such a modulus in the Euclidean case.
For Lipschitz maps into Euclidean spheres, we further determine the
sharp asymptotic.

To state our quantitative results, for $n\ge2$ define
\begin{align}
 \lambda_n&=\inf\{\Lip(h):h:S_{\ell_\infty^n}\to S_{\ell_2^n}
                    \text{ is a homeomorphism}\},\label{eq:lambda}\\
 \mu_n&=\inf\{\Lip(f):f:S_{\ell_\infty^n}\to S_{\ell_2^n},\
                    \deg f\ne0\},\label{eq:mu}
\end{align}
and define Johnson's ball constant by
\begin{equation}\label{eq:Cn}
 C_n=\inf\{\Lip(g):g:B_{\ell_\infty^n}\to B_{\ell_2^n}
                    \text{ is a homeomorphism onto }B_{\ell_2^n}\}.
\end{equation}
Finally, define
\begin{equation}\label{eq:mu-stable}
\begin{split}
 \mu_n^{\mathrm{st}}
 &=\inf_{\substack{m\ge1,\ r>0}}\inf\{\Lip(F):\\
 &\hspace{15mm}
 F:S_{\ell_\infty^n}\times S_{\ell_2^{m+1}}(r)
 \to S_{\ell_2^{n+m}},\ \deg F\ne0\}.
\end{split}
\end{equation}
All maps in these definitions are required to be Lipschitz.
The next theorem answers Gromov's problem and the sphere formulation
associated with Johnson's problem, and determines the sharp
asymptotic behavior of the three sphere quantities.
Throughout the paper, $\log$ denotes the natural logarithm.

\begin{theorem}\label{thm:gromov}
Let $\lambda_n$, $\mu_n$, and $\mu_n^{\mathrm{st}}$, $n\ge2$,
be defined above. Then:
\begin{enumerate}
\item
The three quantities satisfy
\begin{equation}\label{eq:sharp-gromov}
 \lim_{n\to\infty}\frac{\lambda_n}{\log n}
 =\lim_{n\to\infty}\frac{\mu_n}{\log n}
 =\lim_{n\to\infty}\frac{\mu_n^{\mathrm{st}}}{\log n}
 =\frac12.
\end{equation}

\item
For every $n\ge2$, there is an explicit homeomorphism
$\Phi_n:S_{\ell_\infty^n}\to S_{\ell_2^n}$ satisfying
\begin{equation}\label{eq:explicit-upper}
 \Lip(\Phi_n)\le1+\arsinh\sqrt{n-1}.
\end{equation}
This, together with \eqref{eq:sharp-gromov}, gives
$\Lip(\Phi_n)\sim\frac12\log n$ as $n\to\infty$.

\item
For every fixed $r>0$, the infimum of the Lipschitz constants
of all nonzero-degree maps
$S_{\ell_\infty^n}\times S_{\ell_2^2}(r)\to S_{\ell_2^{n+1}}$
is asymptotic to $\frac12\log n$ as $n\to\infty$.
\end{enumerate}
\end{theorem}

The lower bounds in \Cref{thm:gromov} are uniform over all
$m\ge1$ and $r>0$. More generally, \Cref{thm:sharp-parameter}
gives the same bound for the Lipschitz constant in the first
variable when the second sphere factor is replaced by any closed
connected oriented manifold. The proof combines a Poincar\'e
inequality in the first variable with the nonzero degree of the
map. The conclusions of \Cref{thm:gromov} also hold for the
Euclidean product metric.

On $S_{\ell_2^k}(r)$, the chordal and geodesic distances satisfy
$d_{\mathrm{ch}}\le d_{\mathrm{geo}}\le(\pi/2)d_{\mathrm{ch}}$,
uniformly for $k\ge2$ and $r>0$. Hence the negative answer to
\Cref{prob:gromov}, including its stabilized version, remains
valid when the Euclidean spheres carry geodesic distances.
The precise asymptotics in \Cref{thm:gromov} refer to chordal distances.

The sharp sphere estimate also supplies the lower bound for Johnson's
ball problem, while a separate interior construction gives
the matching upper bound.

\begin{theorem}\label{thm:johnson-ball}
Let $\lambda_n$ and $C_n$ be defined by \eqref{eq:lambda} and
\eqref{eq:Cn}, respectively. Then
\[
 \lim_{n\to\infty}\frac{\lambda_n}{\log n}
 =\lim_{n\to\infty}\frac{C_n}{\log n}=\frac12.
\]
More precisely,
\begin{equation}\label{eq:ball-main-bound}
 \lambda_n\le C_n\le 1+\arsinh\sqrt{n-1}+
 \left(1+\frac{n}{2(n-1)}\right)
 \sqrt{1+\arsinh\sqrt{n-1}}.
\end{equation}
\end{theorem}

\subsection{Outline of the proofs}

The proof of \Cref{thm:main} and the lower bounds in
\Cref{thm:gromov,thm:johnson-ball} use the same centering argument.
After extending a nonzero-degree sphere map to the ambient space, we
average its translates. Degree gives a translate with zero mean. A
lower bound for the resulting variance, combined with a Hilbert-valued
Poincar\'e inequality, then yields a lower bound on oscillation. We prove
the parameter version for an additional closed connected oriented
manifold; this includes Gromov's stabilized problem.

For \Cref{thm:main}, the averaging measures are the stationary measures
of finite product Markov chains with dimension-free spectral gaps.
Approximate cubes in the paving spaces of $c_0$, together with a
small-ball estimate uniform over all centers in the paving space, allow
the estimate to be applied to an arbitrary paving. For the sharp bounds in
\Cref{thm:gromov}, we instead use product Laplace measure and
coordinatewise truncation. Explicit coordinatewise homeomorphisms give
the matching sphere upper bound, and a circle factor gives the
stabilized upper bound. The lower bound for $C_n$ follows by restricting
a ball homeomorphism to the boundary; a radial modification of the
sphere construction gives the upper bound on the ball.

\Cref{sec:degree} establishes the degree--Poincar\'e estimate
used throughout the paper. In \Cref{sec:cubes}, we apply it to arbitrary
pavings of $c_0$ and prove \Cref{thm:main}, thereby answering the
Kasparov--Yu problem. \Cref{sec:lower,sec:upper} give
the matching lower and upper bounds in \Cref{thm:gromov}, answering
Gromov's problem, including its stabilized version, and the associated
sphere formulation of Johnson's problem. Finally,
\Cref{sec:johnson-ball} proves \Cref{thm:johnson-ball} and answers
Johnson's question for closed unit balls.

\section{A degree--Poincar\'e obstruction principle}\label{sec:degree}

This section isolates the common obstruction behind our lower bounds. The
nonzero degree of a map gives a translate at which a suitable average
vanishes, and a Poincar\'e inequality then converts the resulting variance
into a lower bound on local oscillation. We formulate the argument with a
parameter manifold from the outset, so that it applies both to rational
Property~$(H)$ and to the stabilized form of Gromov's sphere problem.

\subsection{Vanishing translated averages}\label{sec:centering}

Since we will average only in the first variable, we first fix the
corresponding modulus of continuity. For a normed space $E$, a Hilbert
space $H$, a compact space $Y$, and a continuous map
$F:S_E\times Y\to S_H$, define the modulus of continuity
$\omega_F$ of $F$ in the first variable by
\[
 \omega_F(t)=\sup\{\|F(x,y)-F(x',y)\|_H:
       y\in Y,\ x,x'\in S_E,\ \|x-x'\|_E\le t\}.
\]
We allow $Y$ to be a point. A family $(F_j)_{j\in J}$ of
sphere-valued maps has a common modulus of uniform continuity
in the first variable if there is a nondecreasing function
$\omega:[0,\infty)\to[0,2]$ such that
\[
 \lim_{t\to 0}\omega(t)=0,
 \qquad \omega_{F_j}(t)\le\omega(t)
 \quad(j\in J,\ t\ge0).
\] For the restrictions of the map $\psi$ in the definition
of rational Property~$(H)$, we may take $\omega=\omega_\psi$.

In this section, $E$ is an $n$-dimensional normed space,
$n\ge2$, $Y$ is a closed connected oriented $m$-dimensional manifold, and
$F:S_E\times Y\to S_{\ell_2^{n+m}}$ is continuous with nonzero degree.
When $m=0$, we take $Y$ to be a point. After choosing an orientation
on $E$, we give $S_E=\partial B_E$ the induced boundary orientation
and $S_E\times Y$ the product orientation. The condition
$\deg F\ne0$ is independent of the chosen orientations.

We use a basic fact about degree: a continuous map
$g:S_E(r)\times Y\to S_{\ell_2^{n+m}}$ that extends to
$B_E(r)\times Y$ must have degree zero. We recall the
argument here, since we will use it again in \Cref{sec:sharp}.

Since $Y$ has no boundary,
$\partial(B_E(r)\times Y)=S_E(r)\times Y$. Let
$i:S_E(r)\times Y\hookrightarrow B_E(r)\times Y$ be the
inclusion. We write $H_j(\,\cdot\,;\mathbb Z)$ for singular
homology with integer coefficients, $i_*$ for the homomorphism
induced by $i$, and $[M]$ for the fundamental class of a closed
oriented manifold $M$.

The fundamental class of the boundary is the boundary of the
relative fundamental class of $B_E(r)\times Y$. It therefore
maps to zero under inclusion:
\begin{equation}\label{eq:boundary-class}
 i_*[S_E(r)\times Y]=0
 \quad\text{in }H_{n+m-1}(B_E(r)\times Y;\mathbb Z).
\end{equation}
If $g$ extends to a map $G:B_E(r)\times Y\to S_{\ell_2^{n+m}}$,
then $g=G\circ i$, so
\[
 g_*[S_E(r)\times Y]
 =G_*i_*[S_E(r)\times Y]=0.
\]
By the definition of degree, the left-hand side equals
$(\deg g)[S_{\ell_2^{n+m}}]$, and hence $\deg g=0$.
We refer to \cite[Sections~2.2, 3.3]{Hatcher} for these
standard facts about degree and fundamental classes.

We next extend $F$ to $E\times Y$. For $R>0$, define
\begin{equation}\label{eq:extension}
 \Theta_{F,R}(x,y)=
 \begin{cases}
  \displaystyle\min\{\|x\|_E/R,1\}
       F\!\left(\frac{x}{\|x\|_E},y\right),&x\ne0,\\[2mm]
  0,&x=0.
 \end{cases}
\end{equation}
The factor $\min\{\|x\|_E/R,1\}$ makes this extension continuous
at $x=0$, while its norm is one whenever $\|x\|_E\ge R$.
The following lemma shows that every finite average of its translates
has a zero.

\begin{lemma}\label{lem:centering}
For every finitely supported $E$-valued random variable $Z$ and every
$R>0$, there exist $a\in E$ and $y_0\in Y$ such that
\begin{equation}\label{eq:mean-zero}
 \E\Theta_{F,R}(a+Z,y_0)=0.
\end{equation}
\end{lemma}

\begin{proof}
Fix $R>0$, and define $A:E\times Y\to B_{\ell_2^{n+m}}$ by
\[
 A(a,y)=\E\Theta_{F,R}(a+Z,y),
 \qquad (a,y)\in E\times Y.
\]
Write $\operatorname{supp}Z=\{z\in E:\PP(Z=z)>0\}$ for the
finite support of $Z$, and set
$M=\max\{\|z\|_E:z\in\operatorname{supp}Z\}$.
The map $A$ is continuous, being a finite average of continuous maps.

Let $r>M+R$. To compare $A(rx,y)$ with $F(x,y)$,
we use the following elementary estimate for nonzero vectors
in a normed space:
\begin{equation}\label{eq:normalization}
 \left\|\frac{x}{\|x\|}-\frac{y}{\|y\|}\right\|
 \le\frac{2\|x-y\|}{\max\{\|x\|,\|y\|\}}
 \qquad(x,y\ne0).
\end{equation}
Indeed, if $\|x\|=s\ge t=\|y\|>0$, then the left side is at most
$\|x-y\|/s+(s-t)/s$, and $s-t\le\|x-y\|$.

Since $r>M+R$, we have $\|rx+z\|_E>R$ for every
$x\in S_E$ and $z\in\operatorname{supp}Z$.
By \eqref{eq:normalization},
\begin{equation}\label{eq:at-infinity}
 \sup_{x\in S_E,\,y\in Y}\|A(rx,y)-F(x,y)\|_2
 \le\omega_F(2M/r)\longrightarrow0.
\end{equation}
Since $E$ is finite-dimensional and $Y$ is compact,
$S_E\times Y$ is compact. The continuity of $F$ therefore
implies $\omega_F(t)\to0$ as $t\to0$, which gives
the convergence above as $r\to\infty$.

Choose $r>M+R$ sufficiently large that
\[
 \delta:=\sup_{x\in S_E,\,y\in Y}
 \|A(rx,y)-F(x,y)\|_2<1.
\]
For $0\le t\le1$, we have
\[
 \|(1-t)F(x,y)+tA(rx,y)\|_2\ge1-t\delta>0.
\]
Define $H_t:S_E\times Y\to S_{\ell_2^{n+m}}$ by
\[
 H_t(x,y)=
 \frac{(1-t)F(x,y)+tA(rx,y)}
      {\|(1-t)F(x,y)+tA(rx,y)\|_2}.
\]
The map $(t,x,y)\mapsto H_t(x,y)$ is continuous, and
\[
 H_0(x,y)=F(x,y),
 \qquad
 H_1(x,y)=\frac{A(rx,y)}{\|A(rx,y)\|_2}.
\]
Thus $H_1$ is homotopic to $F$, and homotopy invariance
of degree gives $\deg H_1=\deg F\ne0$.
Under the orientation-preserving identification
$(x,y)\mapsto(rx,y)$, $H_1$ is the normalized restriction
of $A$ to $S_E(r)\times Y$.
If $A$ had no zero in $B_E(r)\times Y$, its normalization
would extend this boundary map over $B_E(r)\times Y$,
contradicting \eqref{eq:boundary-class}.
\end{proof}

\subsection{A quantitative obstruction to uniform continuity}

We now apply \Cref{lem:centering} to a finite average whose weights
form the stationary distribution of a reversible Markov chain.
At the translate supplied by the lemma, the mean is zero, so the
variance equals the average squared norm of the extension. We
estimate this quantity from below using the probability that the
extension has norm one, and from above using a Poincar\'e inequality.
The comparison will give a quantitative lower bound for the modulus
of continuity of $F$.

Let $I$ be a finite set with at least two states, and let
$P=(p_{ij})_{i,j\in I}$ be an irreducible transition matrix on $I$,
reversible with respect to its stationary probability vector
$\pi=(\pi_i)_{i\in I}$. Thus $\pi_i>0$, $\sum_{i\in I}\pi_i=1$, and
\[
 p_{ij}\ge0,\qquad \sum_{j\in I}p_{ij}=1,\qquad
 \pi_i p_{ij}=\pi_j p_{ji}.
\]
In particular, stationarity gives $\sum_{i\in I}\pi_i p_{ij}=\pi_j$.
We define the spectral gap by
\begin{equation}\label{eq:gap}
 \gamma_{_P}=\frac12\inf_{\substack{f:I\to\R\\f\text{ nonconstant}}}
 \frac{\sum_{i,j\in I}\pi_i p_{ij}|f(i)-f(j)|^2}
      {\sum_{i\in I}\pi_i\left|f(i)-\sum_{j\in I}\pi_jf(j)\right|^2}.
\end{equation}
Since the chain is finite and irreducible, $\gamma_{_P}>0$.
For every $f:I\to\R$, the scalar Poincar\'e inequality is
\[
 \sum_{i\in I}\pi_i\left|f(i)-\sum_{j\in I}\pi_jf(j)\right|^2
 \le\frac1{2\gamma_{_P}}\sum_{i,j\in I}\pi_i p_{ij}|f(i)-f(j)|^2.
\]
For an integer $N\ge1$ and a map $h:I\to\ell_2^N$, applying this
inequality to each coordinate function $i\mapsto h(i)_k$,
$1\le k\le N$, and summing gives
\begin{equation}\label{eq:poincare}
 \sum_{i\in I}\pi_i\left\|h(i)-\sum_{j\in I}\pi_jh(j)\right\|_2^2
 \le\frac1{2\gamma_{_P}}\sum_{i,j\in I}\pi_i p_{ij}\|h(i)-h(j)\|_2^2.
\end{equation}
We will use this inequality with $N=n+m$ to estimate the modulus
of continuity of $F$ in the first variable.

\begin{lemma}\label{lem:obstruction}
Under the above assumptions, suppose that $(v_i)_{i\in I}\subset E$
and $D>0$ satisfy
\[
 \|v_i-v_j\|_E\le D
 \qquad(i,j\in I,\ p_{ij}>0).
\]
For $a\in E$ and $R>0$, set $I_R(a)=\{i\in I:\|a+v_i\|_E<R\}$.
Then, for every $R>0$,
\begin{equation}\label{eq:obstruction}
 \omega_F(2D/R)^2
 \ge 2\gamma_{_P}-(8+2\gamma_{_P})
       \sup_{a\in E}\sum_{i\in I_R(a)}\pi_i.
\end{equation}
\end{lemma}

\begin{proof}
Fix $R>0$. We apply \Cref{lem:centering} to the random variable
$i\mapsto v_i$ on $(I,\pi)$. For the resulting $a\in E$ and $y_0\in Y$,
define $h:I\to B_{\ell_2^{n+m}}$ by $h(i)=\Theta_{F,R}(a+v_i,y_0)$.
Then $\sum_{i\in I}\pi_i h(i)=0$.

Write $I_R=I_R(a)$. For $i\in I\setminus I_R$, we have
$\|a+v_i\|_E\ge R$, so \eqref{eq:extension} gives $\|h(i)\|_2=1$.
Therefore,
\begin{equation}\label{eq:variance-lower}
 \sum_{i\in I}\pi_i\|h(i)\|_2^2
 \ge\sum_{i\in I\setminus I_R}\pi_i
 =1-\sum_{i\in I_R}\pi_i.
\end{equation}
We next estimate $\|h(i)-h(j)\|_2$ when $p_{ij}>0$.
If $i,j\in I\setminus I_R$, \eqref{eq:normalization} and the
hypothesis $\|v_i-v_j\|_E\le D$ give
\[
 \left\|\frac{a+v_i}{\|a+v_i\|_E}
       -\frac{a+v_j}{\|a+v_j\|_E}\right\|_E
 \le\frac{2\|v_i-v_j\|_E}{R}
 \le\frac{2D}{R}.
\]
By \eqref{eq:extension} and the definition of $\omega_F$, it follows
that $\|h(i)-h(j)\|_2\le\omega_F(2D/R)$.

For pairs with at least one index in $I_R$, we use
$\|h(i)-h(j)\|_2\le2$, since $h(I)\subset B_{\ell_2^{n+m}}$.
The sum of their weights $\pi_i p_{ij}$ is bounded using stationarity:
\[
 \begin{aligned}
 \sum_{\substack{i,j\in I\\i\in I_R\ \text{or}\ j\in I_R}}\pi_i p_{ij}
 &\le \sum_{i\in I_R}\sum_{j\in I}\pi_i p_{ij}
       +\sum_{j\in I_R}\sum_{i\in I}\pi_i p_{ij}\\
 &=2\sum_{i\in I_R}\pi_i.
 \end{aligned}
\]
Splitting the sum over these two sets of pairs and using
$\sum_{i,j\in I}\pi_i p_{ij}=1$ yields
\begin{equation}\label{eq:energy-upper}
 \sum_{i,j\in I}\pi_i p_{ij}\|h(i)-h(j)\|_2^2
 \le\omega_F(2D/R)^2+8\sum_{i\in I_R}\pi_i.
\end{equation}
We now apply the Poincar\'e inequality \eqref{eq:poincare} to $h$,
whose mean is zero. Substituting \eqref{eq:variance-lower} and
\eqref{eq:energy-upper} gives
\[
 2\gamma_{_P}\left(1-\sum_{i\in I_R}\pi_i\right)
 \le\omega_F(2D/R)^2+8\sum_{i\in I_R}\pi_i.
\]
Rearranging and bounding $\sum_{i\in I_R}\pi_i$ by its supremum
over all translates proves \eqref{eq:obstruction}.
\end{proof}

\Cref{lem:obstruction} gives
\[
 \sup_{a\in E}\sum_{i\in I_R(a)}\pi_i
 \le\frac{\gamma_{_P}}{8+2\gamma_{_P}}
 \quad\Longrightarrow\quad
 \omega_F(2D/R)\ge\sqrt{\gamma_{_P}}.
\]
Suppose this probability bound holds along a family of maps, the
spectral gaps have a common positive lower bound, and $D/R\to0$.
Then the family has no common modulus of uniform continuity in the
first variable.
In the next section, we construct the required chains and points in
sufficiently large spaces of any paving of $c_0$. The probability
bound will hold for every center in the whole paving space.

\section{Obstructions to rational Property $(H)$ in $c_0$}\label{sec:cubes}

The purpose of this section is to give a negative answer to the question of Kasparov and Yu for $c_0$: the space $c_0$ does not have rational Property $(H)$, and hence does not have Property $(H)$. To prove this, we apply the obstruction principle from \Cref{lem:obstruction} to an arbitrary paving of $c_0$. We first construct product chains on finite cubes with uniformly positive spectral gap and short transitions in the supremum norm. We then place uniformly controlled linear images of these cubes in sufficiently large spaces of the paving. A small-ball estimate uniform over all centers in the paving space completes the link with \Cref{lem:obstruction}.

\subsection{Product chains with short transitions}

We begin with a reversible chain on a finite interval whose spectral
gap is bounded below independently of the interval. Its nonnegative
spectrum preserves this bound under tensor products. In the supremum
norm, the product chain can update all coordinates simultaneously
while each transition still has length at most one.

For an integer $r\ge1$, let $I_r=\{-r,\ldots,r\}$ and set
\begin{equation}\label{eq:geometric-law}
 s_r=\sum_{j=-r}^r2^{-|j|}=3-2^{1-r},
 \qquad \pi_r(j)=s_r^{-1}2^{-|j|}\quad(j\in I_r).
\end{equation}
For $i,j\in I_r$ with $|i-j|=1$, define
\begin{equation}\label{eq:chain}
 P_r(i,j)=\frac14\min\{1,\pi_r(j)/\pi_r(i)\}.
\end{equation}
All other off-diagonal entries are zero, and the diagonal entries make
each row sum to one. The chain is irreducible, and
\[
 \pi_r(i)P_r(i,j)=\frac14\min\{\pi_r(i),\pi_r(j)\}
 \qquad(|i-j|=1),
\]
so it is reversible. For an integer $k\ge1$, let
$\mathcal P_{r,k}=P_r^{\otimes k}$ be the product chain on $I_r^k$,
with stationary probability measure $\mu_{r,k}=\pi_r^{\otimes k}$.
Each transition updates all coordinates independently.
We view $I_r^k$ as a subset of $\R^k$ with the supremum norm.

We write $L_2(\pi_r)=L_2(I_r,\pi_r)$ for the Hilbert space of
functions $f:I_r\to\R$, with norm
$\|f\|_{L_2(\pi_r)}=(\sum_{i\in I_r}\pi_r(i)|f(i)|^2)^{1/2}$.

\begin{lemma}\label{lem:product}
For integers $r,k\ge1$, let $P_r$ and $\mathcal P_{r,k}$ be
the chains defined above on $I_r$ and $I_r^k$, respectively.
The operator $P_r$ is positive semidefinite on $L_2(\pi_r)$ and
\begin{equation}\label{eq:product-gap}
 \gamma_{_{P_r}}\ge\frac{(1-2^{-1/2})^2}{4}>\frac1{64}.
\end{equation}
The chain $\mathcal P_{r,k}$ is irreducible and reversible with
stationary probability measure $\mu_{r,k}$. Its spectral gap has the
same lower bound, and for $z,z'\in I_r^k$,
\begin{equation}\label{eq:product-step}
 \mathcal P_{r,k}(z,z')>0
 \quad\Longrightarrow\quad\|z-z'\|_\infty\le1.
\end{equation}
\end{lemma}

\begin{proof}
For every $i\in I_r$, the diagonal entry $P_r(i,i)$ is at least $1/2$.
Thus $2P_r-I$ is a self-adjoint Markov contraction on $L_2(\pi_r)$,
with spectrum contained in $[-1,1]$. The spectrum of $P_r$ therefore
lies in $[0,1]$.

To estimate the spectral gap, let $f:I_r\to\R$. For $1\le t\le r$,
\[
 2^{-t/2}(f(t)-f(0))
 =\sum_{j=1}^t2^{-(t-j)/2}2^{-j/2}(f(j)-f(j-1)).
\]
Extend the sequence $(2^{-j/2}(f(j)-f(j-1)))_{j=1}^r$ by zero
to all integer indices. For $1\le t\le r$, the left side is the
$t$-th coordinate of a sum of shifts of the extended sequence, with
coefficients $2^{-h/2}$, $h\ge0$. Minkowski's inequality in $\ell_2(\mathbb Z)$
and $\sum_{h=0}^{\infty}2^{-h/2}=(1-2^{-1/2})^{-1}$ give, after squaring,
\begin{equation}\label{eq:hardy}
 \sum_{t=1}^r2^{-t}|f(t)-f(0)|^2
 \le\frac1{(1-2^{-1/2})^2}
       \sum_{j=1}^r2^{-j}|f(j)-f(j-1)|^2.
\end{equation}
We apply the same estimate to $t\mapsto f(-t)$. For $1\le j\le r$,
\eqref{eq:chain} gives
\[
 \pi_r(j)P_r(j,j-1)
 =\pi_r(-j)P_r(-j,-j+1)=\frac{2^{-j}}{4s_r}.
\]
Each edge contributes twice to the sum over $I_r\times I_r$.
Since the mean minimizes the squared $L_2(\pi_r)$ distance to
constants, the estimates on the positive and negative indices yield
\[
 \begin{aligned}
 \sum_{i\in I_r}\pi_r(i)
 \left|f(i)-\sum_{j\in I_r}\pi_r(j)f(j)\right|^2
 &\le\sum_{i\in I_r}\pi_r(i)|f(i)-f(0)|^2\\
 &\le\frac{2}{(1-2^{-1/2})^2}
       \sum_{i,j\in I_r}\pi_r(i)P_r(i,j)|f(i)-f(j)|^2.
 \end{aligned}
\]
The definition \eqref{eq:gap} now gives \eqref{eq:product-gap}.

We next pass to $\mathcal P_{r,k}$. Reversibility follows from the
product formula. Irreducibility follows from that of $P_r$, since
its positive diagonal entries allow coordinate paths to be extended
to the same length. Choose an eigenbasis for $P_r$ in $L_2(\pi_r)$,
with eigenvalues
\[
 1=\lambda_0>\lambda_1\ge\cdots\ge\lambda_{2r}\ge0.
\]
By \eqref{eq:gap}, $\lambda_1=1-\gamma_{_{P_r}}$.
Tensor products of this eigenbasis diagonalize $\mathcal P_{r,k}$.
Each nonconstant tensor eigenvector has an eigenvalue that is a
product of numbers in $[0,1]$, at least one of which is at most
$\lambda_1$. Its eigenvalue is therefore at most $\lambda_1$, so the
product chain has the same spectral gap lower bound.
Finally, each coordinate changes by at most one, giving
\eqref{eq:product-step}.
\end{proof}

\subsection{Small-ball estimates for arbitrary centers}

We will apply the product chain inside a finite-dimensional subspace
of a larger normed space. The next estimate allows the center of the
ball to lie anywhere in that larger space.

\begin{lemma}\label{lem:small-ball}
Let $r,k\ge1$ be integers, let $E$ be a normed space, and let
$U:\ell_\infty^k\to E$ be a linear operator satisfying
\[
 \|z\|_\infty\le\|Uz\|_E\qquad(z\in\ell_\infty^k).
\]
Let $Z=(Z_1,\ldots,Z_k)$ be an $I_r^k$-valued random vector with
law $\mu_{r,k}$. Then
\begin{equation}\label{eq:small-ball}
 \sup_{a\in E}\PP(\|a+UZ\|_E<r)
 \le\left(1-\frac{2^{-r}}{3-2^{1-r}}\right)^k.
\end{equation}
\end{lemma}

\begin{proof}
The lower bound on $U$ allows us to recover each coordinate of $z$
from $Uz$ without increasing norms. Indeed, $U$ is injective, and
for each $1\le i\le k$, the functional $Uz\mapsto z_i$ has norm at
most one on $U(\ell_\infty^k)$. Extend it by Hahn--Banach to
$f_i\in E^*$
with $\|f_i\|\le1$.

Fix $a\in E$. If $\|a+UZ\|_E<r$, then
\[
 |f_i(a)+Z_i|<r\qquad(1\le i\le k).
\]
For each $i$, the interval $(-r-f_i(a),r-f_i(a))$ has length $2r$.
It contains at most $2r$ integers and therefore misses at least one
point of $I_r$. Since
$\pi_r(j)\ge 2^{-r}/(3-2^{1-r})$ for every $j\in I_r$, it follows that
\[
 \PP(|f_i(a)+Z_i|<r)\le1-\frac{2^{-r}}{3-2^{1-r}}.
\]
The coordinates of $Z$ are independent, so
\[
 \PP(\|a+UZ\|_E<r)
 \le\prod_{i=1}^k\PP(|f_i(a)+Z_i|<r)
 \le\left(1-\frac{2^{-r}}{3-2^{1-r}}\right)^k.
\]
Since $a$ was arbitrary, taking the supremum of
$\PP(\|a+UZ\|_E<r)$ over $a\in E$ proves \eqref{eq:small-ball}.
\end{proof}

We now use the cubical points supplied by $U$ in
\Cref{lem:obstruction}. Let $E$ be a finite-dimensional normed space,
let $Y$ be a closed connected oriented manifold, possibly a point,
and let
$F:S_E\times Y\to S_{\ell_2^{\dim E+\dim Y}}$ be continuous with nonzero
degree. Suppose that a linear map $U:\ell_\infty^k\to E$, $k\ge2$,
satisfies
\begin{equation}\label{eq:cube-embedding}
 \|z\|_\infty\le\|Uz\|_E\le D\|z\|_\infty
 \qquad(z\in\ell_\infty^k).
\end{equation}
For an integer $r\ge1$, apply \Cref{lem:obstruction} to the chain
$\mathcal P_{r,k}$ on $I_r^k$ and to the points
$v_z=Uz$, $z\in I_r^k$. By \eqref{eq:product-step} and
\eqref{eq:cube-embedding}, every transition moves these points by at
most $D$. Combining \Cref{lem:obstruction,lem:product,lem:small-ball}
with $R=r$ gives
\begin{equation}\label{eq:quantitative-cube}
 \begin{aligned}
 \omega_F(2D/r)^2
 &\ge 2\gamma_{_{\mathcal P_{r,k}}}
 \left[1-\left(1-\frac{2^{-r}}{3-2^{1-r}}\right)^k\right]
 -8\left(1-\frac{2^{-r}}{3-2^{1-r}}\right)^k\\
 &\ge\frac{1-257\left(1-\frac{2^{-r}}{3-2^{1-r}}\right)^k}{32}.
 \end{aligned}
\end{equation}
The bound is independent of the dimension of $E$ and of the parameter
space $Y$. The degree hypothesis concerns $F$ on $S_E\times Y$; the
cubical image $U(\ell_\infty^k)$ is used only to supply the points in
\Cref{lem:obstruction}.

\subsection{Arbitrary pavings of $c_0$}\label{sec:products}

The fact that the center in \eqref{eq:quantitative-cube} may range over
the whole space is essential here.
 A paving of $c_0$ need not contain any coordinate cube exactly.
Density supplies an approximate cube in one paving space, while the
degree hypothesis remains on the whole sphere of that paving space.

\begin{theorem}\label{thm:paving}
Let $(V_j)$ be a paving of $c_0$. For each $j$ with $\dim V_j\ge2$,
let $F_j:S_{V_j}\to S_{\ell_2^{\dim V_j}}$ be a continuous map of
nonzero degree. Then the maps $(F_j)$ have no common modulus of
uniform continuity.
\end{theorem}

\begin{proof}
Suppose that $\omega$ is a common modulus. Choose an integer $r\ge1$
such that $\omega(6/r)<1/8$. Since
$0<1-2^{-r}/(3-2^{1-r})<1$, we may then choose $k\ge2$ such that
\[\left(1-2^{-r}/(3-2^{1-r})\right)^k<\frac{1}{514}.\]

Let $e_1,\ldots,e_k$ be the first $k$ vectors of the canonical
unit vector basis of $c_0$. Since the paving is increasing and has
dense union, we may choose $j$ and $v_1,\ldots,v_k\in V_j$ such that \[\sum_{i=1}^k\|v_i-e_i\|_\infty<\frac12.\]
Define $U:\ell_\infty^k\to V_j$ by $Uz=2\sum_{i=1}^kz_iv_i.$
For every $z\in\ell_\infty^k$,
\[
 \begin{aligned}
 \left\|\frac{Uz}{2}-\sum_{i=1}^kz_ie_i\right\|_\infty
 &=\left\|\sum_{i=1}^kz_i(v_i-e_i)\right\|_\infty\\
 &\le\|z\|_\infty\sum_{i=1}^k\|v_i-e_i\|_\infty
 <\frac12\|z\|_\infty.
 \end{aligned}
\]
Since $\|\sum_{i=1}^kz_ie_i\|_\infty=\|z\|_\infty$, the triangle
inequality now gives
\[
 \|z\|_\infty\le\|Uz\|_\infty\le3\|z\|_\infty.
\]
Thus $U$ is injective, and hence $\dim V_j\ge k$.
Apply \eqref{eq:quantitative-cube} with $E=V_j$, $F=F_j$, $D=3$,
and $Y$ a point. Since $\omega_{F_j}\le\omega$, we obtain
\[
 \omega(6/r)^2
 \ge\omega_{F_j}(6/r)^2
 \ge\frac{1-257\left(1-2^{-r}/(3-2^{1-r})\right)^k}{32}
 >\frac1{64},
\]
contrary to the choice of $r$. The nonzero-degree hypothesis is imposed
on $F_j$ over the whole sphere $S_{V_j}$, while $U$ is used only to
provide the points $Uz$, $z\in I_r^k$, needed in
\Cref{lem:obstruction}.
\end{proof}

\begin{proof}[Proof of \Cref{thm:main}]
Suppose that $c_0$ has rational Property~$(H)$, witnessed by pavings
$(V_j)$ and $(H_j)$ and a map $\psi$. For each $j$, choose a linear
isometry $T_j:H_j\to\ell_2^{\dim V_j}$. Then
\[
 F_j=T_j\circ\psi|_{S_{V_j}}:
 S_{V_j}\to S_{\ell_2^{\dim V_j}}
\]
has nonzero degree. Since the $T_j$ are isometries, the maps $(F_j)$
have the common modulus
$\omega_\psi$, contradicting \Cref{thm:paving}. Thus $c_0$ does not
have rational Property~$(H)$. Since Property~$(H)$ implies rational
Property~$(H)$, $c_0$ has neither rational Property~$(H)$ nor
Property~$(H)$.
\end{proof}

\section{Sharp lower bounds for maps from cubical spheres}\label{sec:lower}

We now turn to the quantitative sphere formulations associated with
the problems of Gromov and Johnson,
including Gromov's stabilized formulation. The finite-chain argument from the
preceding section already forces a fixed oscillation at distances of order
$1/\log n$ and hence gives the correct logarithmic scale. To recover the
sharp constant $1/2$, we replace the finite chains by product Laplace measure
and combine the degree-centering argument with coordinatewise truncation.

\subsection{Uniform continuity under stabilization}

Taking $E=\ell_\infty^n$ in \eqref{eq:quantitative-cube} gives the
following estimate, uniformly over all such parameter manifolds.

\begin{theorem}\label{thm:stable-modulus}
Let $n\ge1024$, and let $Y$ be a closed connected oriented
$m$-dimensional manifold, with $Y$ a point when $m=0$. Every continuous map
$F:S_{\ell_\infty^n}\times Y\to S_{\ell_2^{n+m}}$ of nonzero degree satisfies
\begin{equation}\label{eq:stable-modulus}
 \omega_F\left(\frac{8\log 2}{\log n}\right)>\frac18.
\end{equation}
\end{theorem}

\begin{proof}[Proof of \Cref{thm:stable-modulus}]
Take $E=\ell_\infty^n$, $k=n$, $U=I$, and
$r=\lfloor\tfrac12\log_2 n\rfloor$. For $n\ge1024$, we have
$r\ge\tfrac14\log_2 n$. Since $3-2^{1-r}\le3$,
\[
 \left(1-\frac{2^{-r}}{3-2^{1-r}}\right)^n
 \le\exp\left(-\frac{n2^{-r}}{3-2^{1-r}}\right)
 \le\exp(-\sqrt n/3)<\frac1{514}.
\]
Now \eqref{eq:quantitative-cube} gives $\omega_F(2/r)>1/8$.
Since $2/r\le8\log 2/\log n$, monotonicity of $\omega_F$ proves the claim.
\end{proof}

In particular, for fixed $0<\alpha\le1$, every such map satisfying
$\omega_F(t)\le C t^\alpha$ must have
\[
 C>\frac18\left(\frac{\log n}{8\log 2}\right)^\alpha.
\]
This estimate also allows the parameter manifolds to vary with $n$.

\subsection{Laplace measure and coordinatewise truncation}\label{sec:laplace}

To obtain the sharp asymptotic lower bound, we use product Laplace measure.
We record the Poincar\'e inequality with its constant, since this
constant enters the final asymptotic. Let
$Z=(Z_1,\ldots,Z_n)$ have independent coordinates with density
$\tfrac12e^{-|t|}$. For any Lipschitz map $g:\R^n\to\R^N$,
\begin{equation}\label{eq:laplace-poincare}
 \Var(g(a+sZ))
 \le4s^2\E\sum_{i=1}^n|\partial_i g(a+sZ)|_2^2,
 \qquad a\in\R^n,\quad s>0.
\end{equation}
Here $\Var(V)=\E|V-\E V|_2^2$. The scalar inequality and its
tensorization are standard; see \cite{Ledoux}. To see the constant four,
take a smooth Lipschitz real function $h$. Integration by parts on the
positive half-line and Cauchy--Schwarz give
\[
 \begin{aligned}
 \int_0^\infty |h(t)-h(0)|^2e^{-t}\,dt
 &=2\int_0^\infty (h(t)-h(0))h'(t)e^{-t}\,dt\\
 &\le2\left(\int_0^\infty |h(t)-h(0)|^2e^{-t}\,dt\right)^{1/2}
       \left(\int_0^\infty |h'(t)|^2e^{-t}\,dt\right)^{1/2}.
 \end{aligned}
\]
The negative half-line is treated in the same way. The mean minimizes
squared distance to constants, so this proves the one-dimensional
inequality. Tensorization, followed by summation over the coordinates
of $g$, proves \eqref{eq:laplace-poincare}; translation and scaling give
the displayed form. Approximation extends it to Lipschitz functions.

For a linear map $T:\ell_\infty^n\to\ell_2^N$, let
$\varepsilon_1,\ldots,\varepsilon_n$ be independent uniform signs. Since
$\E(\varepsilon_i\varepsilon_j)=\delta_{ij}$,
\begin{equation}\label{eq:hilbert-derivative}
 \sum_{i=1}^n|Te_i|_2^2
 =\E_\varepsilon|T\varepsilon|_2^2
 \le\|T\|^2.
\end{equation}
Thus $\|T\|$ controls the sum of squared coordinate
derivatives without a dimension factor.

We will also use the following estimate for a homogeneous extension.
It uses only the fact that the values of the sphere map have Euclidean
norm one.

\begin{lemma}\label{lem:radial}
Let $E$ be a finite-dimensional normed space and let
$f:S_E\to S_{\ell_2^N}$ be $L$-Lipschitz. Define $Q:E\to\R^N$ by
$Q(0)=0$ and $Q(x)=\|x\|_E f(x/\|x\|_E)$ for $x\ne0$. Then
\[
 |Q(x)|_2=\|x\|_E,\qquad \Lip(Q)\le\sqrt{4L^2+1}.
\]
\end{lemma}

\begin{proof}
The norm identity follows from the definition. If one of $x$ and $z$
is zero, the Lipschitz estimate follows immediately from this identity.
We may therefore write $x=au$ and $z=bv$, where $u,v\in S_E$ and
$a\ge b>0$. Then $|a-b|\le\|x-z\|_E$ and
$\|u-v\|_E\le2\|x-z\|_E/a$. Since $|f(u)|_2=|f(v)|_2=1$,
\[
 \begin{aligned}
  |Q(x)-Q(z)|_2^2
   &=(a-b)^2+ab|f(u)-f(v)|_2^2\\
   &\le\|x-z\|_E^2+4L^2(b/a)\|x-z\|_E^2\\
   &\le(4L^2+1)\|x-z\|_E^2.
 \end{aligned}
\]
\end{proof}

For an $L$-Lipschitz map from $S_{\ell_\infty^n}$ to a Euclidean unit sphere, we
compose its homogeneous extension with the coordinatewise
truncation $C_R(x)_i=\max\{-1,\min\{x_i/R,1\}\}$.
Outside the cube $[-R,R]^n$, the truncation map takes values in
$S_{\ell_\infty^n}$, so the composition uses the sphere map itself.
Its Lipschitz bound there is $L/R$, rather than
$\sqrt{4L^2+1}/R$. The larger derivative bound is therefore weighted
only by the probability of falling inside the cube.

\subsection{The sharp lower bound}\label{sec:sharp}

The next estimate gives the sharp asymptotic constant. It requires only
a Lipschitz bound in the first variable; the parameter manifold is
otherwise arbitrary. Since coordinatewise truncation changes the extension,
we verify the degree of its averaged boundary map within the proof.

\begin{theorem}\label{thm:sharp-parameter}
Let $n\ge2$, let $Y$ be a closed connected oriented $m$-dimensional
manifold, with $Y$ a point when $m=0$, and let
$F:S_{\ell_\infty^n}\times Y\to S_{\ell_2^{n+m}}$ be continuous with nonzero degree. Suppose
\[
 L=\sup_{y\in Y}\Lip(F(\,\cdot\,,y))<\infty.
\]
For every $R>0$, put $\varepsilon_n(R)=(1-e^{-R})^n$. Then
\begin{equation}\label{eq:sharp-finite}
 L^2\ge
 \frac{\bigl[R^2(1-\varepsilon_n(R))-4\varepsilon_n(R)\bigr]_+}
 {4(1+3\varepsilon_n(R))},
 \qquad [a]_+=\max\{a,0\}.
\end{equation}
Consequently, if $L_n$ denotes the corresponding value of $L$ for a
sequence of such maps, then $\liminf\limits_{n\to\infty}\frac{L_n}{\log n}\ge1/2$,
even when the parameter manifolds vary with $n$.
\end{theorem}

\begin{proof}[Proof of \Cref{thm:sharp-parameter}]
For each $y\in Y$, extend $F(\,\cdot\,,y)$ homogeneously by setting
$Q_y(0)=0$ and
\[
 Q_y(x)=\|x\|_\infty F(x/\|x\|_\infty,y)\qquad(x\ne0).
\]
By \Cref{lem:radial},
\[
 |Q_y(x)|_2=\|x\|_\infty,
 \qquad \Lip(Q_y)\le\sqrt{4L^2+1}.
\]
We now bring the whole space back to the cube by coordinatewise
truncation. Namely, let
\[
 C_R(x)_i=\max\{-1,\min\{x_i/R,1\}\},
 \qquad u_R(x,y)=Q_y(C_R(x)).
\]
The map $u_R$ is jointly continuous and satisfies
\begin{equation}\label{eq:clipped-norm}
 |u_R(x,y)|_2=\min\{\|x\|_\infty/R,1\}.
\end{equation}
For fixed $y$, the global Lipschitz constant of $u_R(\,\cdot\,,y)$ is
at most $\sqrt{4L^2+1}/R$. The point of the truncation is that the
better bound $L/R$ holds outside $[-R,R]^n$. Indeed, there
$C_R(x)\in S_{\ell_\infty^n}$ and
$u_R(x,y)=F(C_R(x),y)$, while $C_R$ is $1/R$-Lipschitz.

Let $Z=(Z_1,\ldots,Z_n)$ have independent coordinates with density
$e^{-|t|}/2$ on $\R$, and put
$p_{a,s}=\PP(\|a+sZ\|_\infty<R)$. Rademacher's theorem applies to
$u_R(\,\cdot\,,y)$. At almost every point outside the cube, its
derivative, viewed from $\ell_\infty^n$ to $\ell_2^{n+m}$, has norm
at most $L/R$; inside the cube, the corresponding bound is
$\sqrt{4L^2+1}/R$. Since the boundary of the cube has probability
zero, \eqref{eq:laplace-poincare} and \eqref{eq:hilbert-derivative}
give
\begin{equation}\label{eq:sharp-variance}
 \begin{aligned}
 \Var(u_R(a+sZ,y))
 &\le\frac{4s^2}{R^2}
   \bigl[L^2(1-p_{a,s})+(4L^2+1)p_{a,s}\bigr]\\
 &=\frac{4s^2}{R^2}
   \bigl[L^2+(3L^2+1)p_{a,s}\bigr]\qquad(s>0).
 \end{aligned}
\end{equation}
No differentiation in the parameter is involved. On the other hand,
\eqref{eq:clipped-norm} gives
$\E|u_R(a+sZ,y)|_2^2\ge1-p_{a,s}$. Subtracting the variance from this
second moment, we obtain
\begin{equation}\label{eq:sharp-mean}
 |\E u_R(a+sZ,y)|_2^2
 \ge1-\frac{4s^2L^2}{R^2}
 -\left(1+\frac{4s^2(3L^2+1)}{R^2}\right)p_{a,s}.
\end{equation}

We next show that the average on the left must vanish at some
translate. Assume first that $R>2L$, and choose $\rho>R$ so large that
\[
 \left(1+\frac{4(3L^2+1)}{R^2}\right)
 \frac12e^{-(\rho-R)}
 <\frac12\left(1-\frac{4L^2}{R^2}\right).
\]
For $x\in S_{\ell_\infty^n}$, $y\in Y$, and $0\le s\le1$, consider
\[
 H_s(x,y)=\E u_R(\rho x+sZ,y).
\]
These maps are jointly continuous by dominated convergence.
We claim that none of them vanishes. Indeed, if $|x_j|=1$, then for
$s>0$ the one-dimensional Laplace tail gives
\[
 p_{\rho x,s}\le\tfrac12e^{-(\rho-R)/s}
 \le\tfrac12e^{-(\rho-R)}.
\]
It follows from \eqref{eq:sharp-mean} and the choice of $\rho$ that
\[
 \begin{aligned}
 |H_s(x,y)|_2^2
 &\ge 1-\frac{4L^2}{R^2}
 -\left(1+\frac{4(3L^2+1)}{R^2}\right)
       \frac12e^{-(\rho-R)}\\
 >\frac12\left(1-\frac{4L^2}{R^2}\right)>0.
 \end{aligned}
\]
The same conclusion is immediate for $s=0$, since $H_0(x,y)$ has
norm one. We may therefore normalize $H_s$ throughout the homotopy.
At its initial point,
\[
 \frac{H_0(x,y)}{|H_0(x,y)|_2}=F(C_R(\rho x),y).
\]
This map has the same degree as $F$. To see this, let $\phi_t(x)$ be
the coordinatewise truncation to $[-1,1]$ of
$\bigl(1+t(\rho/R-1)\bigr)x$, $0\le t\le1$.
At least one coordinate of $x$ has absolute value one, so every
$\phi_t(x)$ lies in $S_{\ell_\infty^n}$. Moreover,
$\phi_0(x)=x$ and $\phi_1(x)=C_R(\rho x)$. Thus $F(\phi_t(x),y)$
joins $F$ to the initial normalized average. The normalized maps
$H_s/|H_s|_2$ then give a second homotopy, and hence
\[
 \deg\left(\frac{H_1}{|H_1|_2}\right)=\deg F\ne0.
\]

This boundary map cannot extend over
$B_{\ell_\infty^n}(\rho)\times Y$. But the map
$A(a,y)=\E u_R(a+Z,y)$ is continuous by dominated convergence, and
its restriction to the boundary, after the identification
$(x,y)\mapsto(\rho x,y)$, is $H_1$. If $A$ never vanished, its
normalization would provide precisely such an extension, contradicting
\eqref{eq:boundary-class}. We have therefore found $(a_0,y_0)$ such
that
\begin{equation}\label{eq:sharp-zero}
 \E u_R(a_0+Z,y_0)=0.
\end{equation}

It remains to estimate how often a translate of $Z$ can fall inside
the cube. Among intervals of length $2R$, the centered interval
$(-R,R)$ has the largest Laplace measure, because the density is
symmetric and nonincreasing away from zero. Independence therefore
gives
\begin{equation}\label{eq:laplace-smallball}
 p_{a,1}\le(1-e^{-R})^n=\varepsilon_n(R)\qquad(a\in\R^n).
\end{equation}
At the zero in \eqref{eq:sharp-zero}, the second moment is the
variance. Hence \eqref{eq:clipped-norm} and
\eqref{eq:sharp-variance} give
\[
 1-p_{a_0,1}\le\E|u_R(a_0+Z,y_0)|_2^2
 \le\frac4{R^2}\bigl[L^2+(3L^2+1)p_{a_0,1}\bigr].
\]
Rearranging,
\[
 L^2\ge
 \frac{R^2(1-p_{a_0,1})-4p_{a_0,1}}
      {4(1+3p_{a_0,1})}.
\]
The last quotient decreases as $p_{a_0,1}$ increases. Together with
\eqref{eq:laplace-smallball} and $L^2\ge0$, this proves
\eqref{eq:sharp-finite} when $R>2L$. When $R\le2L$, the right-hand
side of \eqref{eq:sharp-finite} is at most $R^2/4\le L^2$, so the same
estimate remains valid.

To obtain the asymptotic statement, fix $0<\delta<1$ and take
$R=(1-\delta)\log n$. Then
$\varepsilon_n(R)\le e^{-n^\delta}\to0$, and \eqref{eq:sharp-finite}
gives
\[
 \liminf_{n\to\infty}\frac{L_n^2}{(\log n)^2}
 \ge\frac{(1-\delta)^2}{4}.
\]
Taking square roots and then letting $\delta\to 0$ completes the
proof.
\end{proof}

\section{Matching upper bounds and stabilization}\label{sec:upper}

It remains to show that the lower bound from \Cref{sec:lower} is sharp. We
first construct explicit homeomorphisms
$S_{\ell_\infty^n}\to S_{\ell_2^n}$ with Lipschitz constant
$\frac12\log n+O(1)$. We then use a circle parameter to obtain the same
asymptotic upper bound in Gromov's stabilized problem.

\subsection{Explicit homeomorphisms of spheres}

We begin by constructing the homeomorphisms in \Cref{thm:gromov}. The
choice of the scalar function below balances the derivatives of its
coordinates against the norm used for normalization. Fix $n\ge2$ and put
\[
 a=\arsinh\!\bigl(\sqrt{n-1}\bigr).
\]
Here $\arsinh$, $\sinh$, and $\cosh$ denote the inverse hyperbolic sine,
the hyperbolic sine, and the hyperbolic cosine, respectively.
Define an odd, continuous, strictly increasing function $g:\R\to\R$
by
\[
 g(t)=
 \begin{cases}
  \displaystyle\frac{\sinh\bigl((a+1)t\bigr)}{\sqrt{n-1}},
      &0\le t\le\dfrac{a}{a+1},\\[6pt]
  \cosh\bigl((a+1)t-a\bigr),&t\ge\dfrac{a}{a+1},
 \end{cases}
 \qquad g(-t)=-g(t).
\]
The two formulas agree at $a/(a+1)$, where both equal one. For $z\ne0$,
set
\begin{equation}\label{eq:upper-map}
 \Phi_n(z)=\frac{(g(z_1),\ldots,g(z_n))}
                   {\bigl(\sum_i g(z_i)^2\bigr)^{1/2}}.
\end{equation}
Although we only need the restriction of $\Phi_n$ to
$S_{\ell_\infty^n}$, we define it on $\R^n\setminus\{0\}$ because the
line segment joining two points of the sphere may leave the sphere.
This allows us to estimate their images by integrating the derivative
of $\Phi_n$ along that segment.

\begin{lemma}\label{lem:upper-map}
The map $\Phi_n:S_{\ell_\infty^n}\to S_{\ell_2^n}$ is a homeomorphism
and satisfies
\[
 \Lip(\Phi_n)\le 1+\arsinh\!\bigl(\sqrt{n-1}\bigr).
\]
\end{lemma}

\begin{proof}
We first see directly why normalization does not lose any information.
The function $g$ is an odd increasing homeomorphism of $\R$, and
$g(1)=\cosh(1)$. Given $y\in S_{\ell_2^n}$, define
\[
 \Phi_n^{-1}(y)_i
 =g^{-1}\!\left(\cosh(1)\,\frac{y_i}{\|y\|_\infty}\right).
\]
The largest absolute value among the arguments of $g^{-1}$ is
$\cosh(1)=g(1)$, so this formula produces a point of
$S_{\ell_\infty^n}$. Moreover,
\[
 g\bigl(\Phi_n^{-1}(y)_i\bigr)
 =\cosh(1)\,\frac{y_i}{\|y\|_\infty}.
\]
The vector on the right has Euclidean norm
$\cosh(1)/\|y\|_\infty$, and normalization therefore returns $y$.
Thus the displayed formula is indeed the
inverse of $\Phi_n$; it also makes its continuity clear.

It remains to estimate the Lipschitz constant. The two pieces of $g$
were chosen so that, away from their joining points,
\[
 \frac{g'(t)^2}{(a+1)^2}-g(t)^2=
 \begin{cases}
  (n-1)^{-1},&|t|<\dfrac{a}{a+1},\\[4pt]
  -1,&|t|>\dfrac{a}{a+1}.
 \end{cases}
\]
Now suppose that $g$ is differentiable at every coordinate of $z$ and
$\|z\|_\infty>a/(a+1)$. At least one coordinate then contributes $-1$
to the sum of the preceding identities, while each of the remaining
coordinates contributes at most $(n-1)^{-1}$. The negative term
therefore cancels all the positive ones, and we obtain
\[
 \sum_i g'(z_i)^2\le(a+1)^2\sum_i g(z_i)^2.
\]
Euclidean normalization at a nonzero vector $w$ has derivative of norm
$1/\|w\|_2$. Consequently, for every $v\in\R^n$,
\[
 \|D\Phi_n(z)v\|_2
 \le
 \frac{\bigl(\sum_i g'(z_i)^2v_i^2\bigr)^{1/2}}
      {\bigl(\sum_i g(z_i)^2\bigr)^{1/2}}
 \le(a+1)\|v\|_\infty.
\]
Thus $\|D\Phi_n(z)\|\le a+1$, with the derivative viewed from
$\ell_\infty^n$ to $\ell_2^n$.

We finish by carrying this differential estimate back to the sphere.
For $x,y\in S_{\ell_\infty^n}$, every point of the segment joining
them satisfies
\[
 \|(1-t)x+ty\|_\infty
 \ge\max\{1-t\|x-y\|_\infty,
           1-(1-t)\|x-y\|_\infty\}
 \ge1-\frac12\|x-y\|_\infty.
\]
If $\|x-y\|_\infty<2/(a+1)$, the whole segment lies in
\[
 \left\{z:\|z\|_\infty>\frac{a}{a+1}\right\}.
\]
Along this segment, $\Phi_n$ is absolutely continuous and the preceding
derivative estimate holds almost everywhere: a nonconstant affine
coordinate meets either joining value only once, while a coordinate
that stays at a joining value has zero derivative. Integration gives
\[
 \|\Phi_n(x)-\Phi_n(y)\|_2\le(a+1)\|x-y\|_\infty.
\]
When $\|x-y\|_\infty\ge2/(a+1)$, the diameter of the Euclidean sphere
gives the same conclusion:
\[
 \|\Phi_n(x)-\Phi_n(y)\|_2\le2
 \le(a+1)\|x-y\|_\infty.
\]
Since $a+1=1+\arsinh\!\bigl(\sqrt{n-1}\bigr)$, the proof is complete.
\end{proof}

\subsection{Stabilization by a circle}

We now pass from homeomorphisms to stabilized maps of nonzero degree.
A single parameter circle suffices for the upper bound.

\begin{lemma}\label{lem:circle}
Let $n\ge2$ and let $h:S_{\ell_\infty^n}\to S_{\ell_2^n}$ be a
Lipschitz homeomorphism.
For every $r>0$ there is a map
$G_r:S_{\ell_\infty^n}\times S_{\ell_2^2}(r)\to S_{\ell_2^{n+1}}$
with $|\deg G_r|=1$ and
\begin{equation}\label{eq:circle-bound}
 \Lip(G_r)\le\Lip(h)+\frac{\pi}{2r}.
\end{equation}
Its Lipschitz constant in the first variable is exactly $\Lip(h)$.
\end{lemma}

\begin{proof}
Fix $x_0\in S_{\ell_2^n}$. For $0\le t\le2\pi$, define
\[
 G_r\bigl(x,r(\cos t,\sin t)\bigr)=
 \begin{cases}
  (\sin t\,h(x),\cos t),&0\le t\le\pi,\\
  (-\sin t\,x_0,\cos t),&\pi\le t\le2\pi.
 \end{cases}
\]
The two formulas agree when $t=\pi$, and they take the same value at
$t=0$ and $t=2\pi$. Since both have Euclidean norm one, they define a
continuous map into $S_{\ell_2^{n+1}}$.

The second semicircle closes the map without cancelling the degree
carried by $h$. To see this, choose
$z\in S_{\ell_2^n}\setminus\{x_0\}$ and consider the equatorial point
$(z,0)$. Its last coordinate forces any preimage to lie over
$t=\pi/2$ or $t=3\pi/2$. The latter point maps to $(x_0,0)$, while over
the former the map is $(x,r(0,1))\mapsto(h(x),0)$. Thus $(z,0)$ has
the unique preimage
\[
 \bigl(h^{-1}(z),r(0,1)\bigr).
\]
Near this preimage, $\cos t$ is a local coordinate in the last
component, and $h$ is a homeomorphism in the other components. The
local degree is therefore $\pm1$. Since there are no other preimages,
the local-degree formula gives $|\deg G_r|=1$.

The metric control is built into the same formula. When $t$ is fixed,
the dependence on $x$ is multiplied by $|\sin t|$, so its Lipschitz
constant is at most $\Lip(h)$. At $t=\pi/2$ we recover
$x\mapsto(h(x),0)$, and hence the constant in the first variable is
exactly $\Lip(h)$.

Now fix $x$ and move along the parameter circle. On each semicircle
the image curve has speed one in the angular variable, and the two
pieces meet continuously. If $u,u'\in S_{\ell_2^2}(r)$, following the
shorter arc from $u$ to $u'$ changes the angular variable by its arc
length divided by $r$. Since that length is at most
$\frac\pi2\|u-u'\|_2$, we obtain
\[
 \|G_r(x,u)-G_r(x,u')\|_2
 \le\frac\pi{2r}\|u-u'\|_2.
\]
Finally, changing the sphere variable and then the circle variable
yields
\[
 \|G_r(x,u)-G_r(x',u')\|_2
 \le\Lip(h)\|x-x'\|_\infty+\frac\pi{2r}\|u-u'\|_2,
\]
which proves \eqref{eq:circle-bound} for the maximum product distance.
\end{proof}

\subsection{Proof of the sharp asymptotic theorem}

We now combine the lower bound with the two constructions.

\begin{proof}[Proof of \Cref{thm:gromov}]
The homeomorphisms constructed in \Cref{lem:upper-map} satisfy
\begin{equation}\label{eq:matching-upper}
 \Lip(\Phi_n)\le1+\arsinh\!\bigl(\sqrt{n-1}\bigr)
 =\tfrac12\log n+O(1).
\end{equation}
Since every homeomorphism between oriented spheres has degree $\pm1$,
this gives the required upper bound for both $\lambda_n$ and $\mu_n$.
Applying \Cref{lem:circle} with $h=\Phi_n$ gives, for every $r>0$,
\[
 \mu_n^{\mathrm{st}}
 \le\Lip(\Phi_n)+\frac\pi{2r}.
\]
Letting $r\to\infty$ and using \eqref{eq:matching-upper}, we obtain
\[
 \limsup_{n\to\infty}
 \frac{\max\{\lambda_n,\mu_n,\mu_n^{\mathrm{st}}\}}{\log n}
 \le\frac12.
\]

The lower-bound theorem from the preceding section now fits all three
quantities. Its finite estimate depends only on $n$ and $R$, not on
the map or the parameter manifold. When $Y$ is a point, we may take
the infimum over all nonzero-degree maps and obtain
$\liminf_{n\to\infty}\mu_n/\log n\ge1/2$. Since
$\mu_n\le\lambda_n$, the same lower bound follows for $\lambda_n$.
For stabilized maps, the same uniformity allows us to take the
infimum first over the maps and then over $m$ and $r$, which gives
$\liminf_{n\to\infty}\mu_n^{\mathrm{st}}/\log n\ge1/2$. This proves
\eqref{eq:sharp-gromov}.

If $r>0$ is fixed, the circle construction differs from
\eqref{eq:matching-upper} only by the constant $\pi/(2r)$. Together
with the same lower bound, this proves the third assertion of the
theorem. Finally, the upper estimate also holds for the Euclidean
product distance, since that distance dominates the maximum product
distance.
\end{proof}

\section{Johnson's ball problem}\label{sec:johnson-ball}

We now return to \Cref{prob:johnson}. The lower bound is already
contained in the sphere theorem. Indeed, every homeomorphism
$g:B_{\ell_\infty^n}\to B_{\ell_2^n}$ maps the interior onto the
interior and hence the boundary onto the boundary; this follows by
applying invariance of domain to $g$ and to $g^{-1}$. Its restriction
to $S_{\ell_\infty^n}$ is therefore a homeomorphism onto
$S_{\ell_2^n}$ with no larger Lipschitz constant. Consequently,
\begin{equation}\label{eq:lambda-C}
 \lambda_n\le C_n.
\end{equation}

The matching upper bound requires control in the interior of the
cube. It is provided by the following modification of the
coordinatewise hyperbolic construction.

\begin{lemma}\label{lem:ball-upper}
For every $n\ge2$, there is an explicit homeomorphism
$g_n:B_{\ell_\infty^n}\to B_{\ell_2^n}$ such that
\begin{equation}\label{eq:ball-upper}
 \Lip(g_n)\le 1+\arsinh\sqrt{n-1}
 +\left(1+\frac{n}{2(n-1)}\right)
 \sqrt{1+\arsinh\sqrt{n-1}}.
\end{equation}
In particular,
\[
 \Lip(g_n)\le\frac12\log n+O(\sqrt{\log n}).
\]
\end{lemma}

\begin{proof}
Set
\begin{align*}
 L_n&=1+\arsinh\sqrt{n-1}
      +\sqrt{1+\arsinh\sqrt{n-1}},\\
 \rho_n&=\frac{\arsinh\sqrt{n-1}}{L_n}
      =1-\frac1{\sqrt{1+\arsinh\sqrt{n-1}}}.
\end{align*}
Thus \(0<\rho_n<1\) and
\(L_n\rho_n=\arsinh\sqrt{n-1}\).
Define an odd, continuous, strictly increasing function
\(\alpha_n:[-1,1]\to\R\) by
\begin{equation}\label{eq:ball-g}
 \alpha_n(t)=
 \begin{cases}
  \displaystyle\frac{\sinh(L_nt)}{\sqrt{n-1}},
      &0\le t\le\rho_n,\\[6pt]
  \cosh\bigl(L_n(t-\rho_n)\bigr),
      &\rho_n\le t\le1,
 \end{cases}
 \qquad \alpha_n(-t)=-\alpha_n(t).
\end{equation}
The two formulas agree at \(\rho_n\), where both equal one. Define
\begin{equation}\label{eq:ball-b}
 b_n(s)=
 \begin{cases}
  \displaystyle\frac n{n-1},&0\le s\le\rho_n,\\[6pt]
  \displaystyle\frac n{n-1}\frac{1-s}{1-\rho_n},
      &\rho_n\le s\le1.
 \end{cases}
\end{equation}
Thus \(b_n\) is continuous and nonincreasing, positive on
\([0,1)\), and \(b_n(1)=0\). For
\(x\in B_{\ell_\infty^n}\), define
\begin{equation}\label{eq:ball-map}
 g_n(x)=
 \frac{(\alpha_n(x_1),\ldots,\alpha_n(x_n))}
 {\sqrt{\sum_{i=1}^n\alpha_n(x_i)^2+b_n(\|x\|_\infty)}}.
\end{equation}
If \(\|x\|_\infty\le\rho_n\), the square of the denominator is at
least \(n/(n-1)\). If \(\|x\|_\infty\ge\rho_n\), one coordinate has
absolute value \(\|x\|_\infty\), and hence
\(\sum_i\alpha_n(x_i)^2\ge1\). Thus \eqref{eq:ball-map} is well
defined. Since \(b_n\) is positive on \([0,1)\) and vanishes at \(1\),
interior points map into the open Euclidean ball and boundary points
map into its unit sphere.

We next prove that \(g_n\) is a homeomorphism. Fix
\(y\in S_{\ell_2^n}\) and, for
\(0\le t\le\alpha_n(1)/\|y\|_\infty\), set
\[
 x_y(t)=\bigl(\alpha_n^{-1}(ty_1),\ldots,
              \alpha_n^{-1}(ty_n)\bigr).
\]
Then
\(\|x_y(t)\|_\infty=\alpha_n^{-1}(t\|y\|_\infty)\) increases
continuously from zero to one, and
\[
 g_n(x_y(t))
 =\frac{t}{\sqrt{t^2+b_n(\|x_y(t)\|_\infty)}}\,y.
\]
The scalar coefficient on the right increases continuously from zero
to one. It is strictly increasing because, for \(0<s<t\) in the above
interval,
\[
 \frac{b_n(\|x_y(t)\|_\infty)}{t^2}
 \le\frac{b_n(\|x_y(s)\|_\infty)}{t^2}
 <\frac{b_n(\|x_y(s)\|_\infty)}{s^2}.
\]
Thus every point of the radial segment \([0,y]\) has a unique
preimage. Since \(y\) was arbitrary and zero is the unique preimage of
the origin, \(g_n\) is a continuous bijection from the compact space
\(B_{\ell_\infty^n}\) onto the Hausdorff space
\(B_{\ell_2^n}\), and hence is a homeomorphism.

It remains to estimate its Lipschitz constant. For
\(T(w,b)=w/\sqrt{\|w\|_2^2+b}\), differentiation with
respect to \(w\) and \(b\), respectively, gives
\begin{equation}\label{eq:ball-T}
 \|D_wT(w,b)\|_{2\to2}\le
 \frac1{\sqrt{\|w\|_2^2+b}},
 \qquad
 \|\partial_bT(w,b)\|_2
 =\frac{\|w\|_2}{2(\|w\|_2^2+b)^{3/2}}.
\end{equation}
For \(\|x\|_\infty<\rho_n\), we have
\[
 \alpha_n'(x_i)^2
 =L_n^2\left(\alpha_n(x_i)^2+\frac1{n-1}\right),
 \qquad 1\le i\le n.
\]
Since \(b_n(\|x\|_\infty)=n/(n-1)\), summing over \(i\) and using
\eqref{eq:ball-T} yields
\begin{equation}\label{eq:ball-inner}
 \|Dg_n(x)v\|_2\le L_n\|v\|_\infty.
\end{equation}

For \(|t|<1\) and \(|t|\ne\rho_n\),
\[
 \frac{\alpha_n'(t)^2}{L_n^2}-\alpha_n(t)^2=
 \begin{cases}
  (n-1)^{-1},&|t|<\rho_n,\\
  -1,&|t|>\rho_n.
 \end{cases}
\]
Consequently, for almost every \(x\) with
\(\rho_n<\|x\|_\infty<1\), at least one coordinate contributes
\(-1\), and hence
\begin{equation}\label{eq:ball-outer-sum}
 \sum_{i=1}^n\alpha_n'(x_i)^2
 \le L_n^2\sum_{i=1}^n\alpha_n(x_i)^2.
\end{equation}
For \(\rho_n<\|x\|_\infty<1\), we also have
\[
 \sum_{i=1}^n\alpha_n(x_i)^2\ge1,
 \qquad
 |b_n'(\|x\|_\infty)|
 =\frac{n}{n-1}\sqrt{1+\arsinh\sqrt{n-1}}.
\]
Moreover, the derivative of \(x\mapsto\|x\|_\infty\), whenever it
exists, has norm at most one. The chain rule,
\eqref{eq:ball-T}, and \eqref{eq:ball-outer-sum} give, for almost every
\(x\) with \(\rho_n<\|x\|_\infty<1\),
\begin{equation}\label{eq:ball-outer}
 \|Dg_n(x)v\|_2
 \le\left(L_n+\frac{n}{2(n-1)}
 \sqrt{1+\arsinh\sqrt{n-1}}\right)\|v\|_\infty.
\end{equation}
Indeed, the factor multiplying \(|b_n'(\|x\|_\infty)|/2\) is at most
one, since
\[
 \frac{\sqrt{\sum_i\alpha_n(x_i)^2}}
 {\left(\sum_i\alpha_n(x_i)^2+
 b_n(\|x\|_\infty)\right)^{3/2}}\le1.
\]
The exceptional sets have Lebesgue measure zero. Since \(g_n\) is
locally Lipschitz, \eqref{eq:ball-inner} and \eqref{eq:ball-outer},
followed by mollification on the open cube and integration along line
segments, yield
\[
 \Lip(g_n)\le L_n+\frac{n}{2(n-1)}
 \sqrt{1+\arsinh\sqrt{n-1}}.
\]
The estimate extends to the closed cube by continuity and is precisely
\eqref{eq:ball-upper}. Finally,
\(\arsinh\sqrt{n-1}=\frac12\log n+O(1)\), which gives the asserted
asymptotic estimate.
\end{proof}

The restriction of $g_n$ to $S_{\ell_\infty^n}$ is also an explicit
homeomorphism onto $S_{\ell_2^n}$. Indeed, $b_n(1)=0$, so this
restriction is obtained by applying $\alpha_n$ coordinatewise and then
normalizing in the Euclidean norm. The derivative estimate
\eqref{eq:ball-outer-sum} and the line-segment argument from
\Cref{lem:upper-map} give
\[
 \Lip\bigl(g_n|_{S_{\ell_\infty^n}}\bigr)\le L_n
 =1+\arsinh\sqrt{n-1}
  +\sqrt{1+\arsinh\sqrt{n-1}}.
\]
Thus the ball construction also gives the optimal leading constant
$1/2$ in the sphere problem. We retain the construction in
\Cref{lem:upper-map} because it gives the sharper bound
$1+\arsinh\sqrt{n-1}$.

\begin{proof}[Proof of \Cref{thm:johnson-ball}]
As observed above, restriction to the boundary gives
\[
 \lambda_n\le C_n.
\]
On the other hand, \Cref{lem:ball-upper} gives
\[
 C_n\le 1+\arsinh\sqrt{n-1}
 +\left(1+\frac{n}{2(n-1)}\right)
 \sqrt{1+\arsinh\sqrt{n-1}}.
\]
Since \Cref{thm:gromov} gives
$\lambda_n/\log n\to1/2$ and
$\arsinh\sqrt{n-1}=\frac12\log n+O(1)$, these two estimates yield
\[
 \lim_{n\to\infty}\frac{C_n}{\log n}=\frac12.
\]
\end{proof}

\section*{Acknowledgements}

The authors are particularly grateful to Professors Guoliang Yu and Chunlan Jiang
for their longstanding interest in and support of their work on Property~$(H)$.

\end{document}